\documentclass{amsart}

\usepackage{enumerate, amsmath, amsfonts, amssymb, amsthm, wasysym, graphics, graphicx, xcolor, url, hyperref, hypcap, ifthen, xargs}
\hypersetup{colorlinks=true, citecolor=darkblue, linkcolor=darkblue}
\usepackage[all]{xy}
\usepackage{tikz}\usetikzlibrary{trees,shapes,arrows,matrix,calc}
\graphicspath{{figures/}}

\title{Vertex barycenter of generalized associahedra}

\author{Vincent Pilaud}
\address{CNRS \& LIX, \'Ecole Polytechnique, Palaiseau}
\email{vincent.pilaud@lix.polytechnique.fr}
\urladdr{http://www.lix.polytechnique.fr/~pilaud/}

\author{Christian Stump}
\address{Institut f\"ur Mathematik, Freie Universit\"at Berlin, Germany}
\email{christian.stump@fu-berlin.de}
\urladdr{http://homepage.univie.ac.at/christian.stump/}

\thanks{V.\,P.~was supported by the spanish MICINN grant MTM2011-22792, by the French ANR grant EGOS 12 JS02 002 01, and by the European Research Project ExploreMaps (ERC StG 208471).}

\newtheorem{theorem}{Theorem}[section]
\newtheorem{corollary}[theorem]{Corollary}
\newtheorem{proposition}[theorem]{Proposition}
\newtheorem{lemma}[theorem]{Lemma}

\theoremstyle{definition}
\newtheorem{example}[theorem]{Example}

\newcommand{\R}{\mathbb{R}} 
\newcommand{\cN}{\mathcal{N}} 
\newcommand{\fS}{\mathfrak{S}} 
\newcommand{\sfr}{{\sf r}} 
\newcommand{\sfw}{{\sf w}} 
\newcommand{\sfB}{{\sf B}} 

\newcommand{\set}[2]{\left\{ #1 \;\middle|\; #2 \right\}} 
\newcommand{\bigset}[2]{\big\{ #1 \;|\; #2 \big\}} 
\newcommand{\multiset}[2]{\left\{\!\!\left\{ #1 \;\middle|\; #2 \right\}\!\!\right\}} 
\newcommand{\gen}[1]{\left\langle #1 \right\rangle} 
\newcommand{\ssm}{\smallsetminus} 
\newcommand{\dotprod}[2]{\langle #1 | #2 \rangle} 
\newcommand{\one}{\mathbf{1}} 
\newcommand{\eqdef}{\mbox{\,\raisebox{0.2ex}{\scriptsize\ensuremath{\mathrm:}}\ensuremath{=}\,}} 

\newcommand{\fundamentalChamber}{\mathcal{C}} 
\newcommandx{\Perm}[2][1={}, 2=W]{{\sf Perm}^{#1}(#2)} 
\newcommandx{\Asso}[3][1={}, 2=c, 3=W]{{\sf Asso}^{#1}_{#2}(#3)} 
\newcommand{\basepoint}{u} 

\newcommand{\subwordComplex}[1][\Q,\rho]{\mathcal{SC}(#1)} 
\newcommand{\clusterComplex}[1][\sq{c}]{\mathcal{SC}_{#1}} 

\newcommandx{\brickPolytope}[2][1={}, 2=\Q]{\mathcal{B}^{#1}(#2)} 
\newcommandx{\clusterPolytope}[2][1={}, 2=\sq{c}]{\mathcal{B}^{#1}_{#2}} 

\newcommandx{\brickVector}[3][1={}, 2=\Q, 3=I]{\sfB^{#1}_{#2}(#3)} 
\newcommand{\wordprod}[2]{\Pi#1_{#2}} 
\newcommandx{\Root}[4][1={}, 2=\Q, 3=I, 4=k]{{\sfr}^{#1}_{#2}(#3,#4)} 
\newcommandx{\Weight}[4][1={}, 2=\Q, 3=I, 4=k]{{\sfw}^{#1}_{#2}(#3,#4)} 
\newcommand{\conjugation}{\varphi} 
\newcommand{\loday}{\mathsf{L}} 
\newcommand{\orbit}{\mathcal{O}} 

\newcommand{\sq}[1]{{\rm #1}} 
\newcommand{\Q}{\sq{Q}} 
\newcommand{\q}{\sq{q}} 
\newcommand{\w}{\sq{w}} 
\newcommand{\commute}[1]{\ifthenelse{\equal{#1}{}}{\pi}{{\pi(#1)}}} 
\newcommand{\permutation}[1]{\ifthenelse{\equal{#1}{}}{\psi}{\psi(#1)}} 
\newcommand{\rotate}[1]{\ifthenelse{\equal{#1}{}}{\circlearrowleft}{{#1^\circlearrowleft}}} 
\newcommand{\rotation}[1]{\ifthenelse{\equal{#1}{}}{\rho}{\rho(#1)}} 
\newcommand{\shift}[1]{{\sf shift_{\sq{#1}}}} 

\newcommand{\newleftrightarrow}{\mathbin{\tikz [baseline=0.2em] \draw [<->] (-.4em,.4em) -- (.4em,.4em);}} 
\newcommand{\newupdownarrow}{\mathbin{\tikz [thin, baseline=-0.2em] \draw [<->] (0em,-.3em) -- (0em,.3em);}} 
\newcommand{\newupdownleftrightarrow}{\mathbin{\tikz [baseline=-0.2em] {\draw [<->] (0em,-.3em) -- (0em,.3em); \draw [<->] (-.5em,0em) -- (.5em,0em);}}} 
\DeclareRobustCommand{\conjugate}[1]{\ifthenelse{\equal{#1}{}}{\psi}{#1^{\!\ \newupdownarrow\!\ }}} 
\DeclareRobustCommand{\reverse}[1]{\ifthenelse{\equal{#1}{}}{{\newleftrightarrow}}{#1^{\newleftrightarrow}}} 
\DeclareRobustCommand{\conjugatereverse}[1]{\ifthenelse{\equal{#1}{}}{{\newupdownleftrightarrow}}{#1^{\newupdownleftrightarrow}}} 
\newcommand{\mirror}[1]{\ifthenelse{\equal{#1}{}}{\mu}{\mu(#1)}} 

\newcommand{\sw}[2]{\sq{#1}(\sq{#2})} 
\newcommand{\cwo}[1]{\sw{w_\circ}{#1}} 
\newcommand{\cw}[1]{\sq{#1}\cwo{#1}} 
\newcommandx{\translation}[2][1={}, 2=c]{\Omega^{#1}_{#2}} 

\newcommand{\positiveFacet}{\mathsf{P}} 
\newcommand{\negativeFacet}{\mathsf{N}} 

\DeclareMathOperator{\conv}{conv} 
\DeclareMathOperator{\vect}{vect} 
\DeclareMathOperator{\barycenter}{\mathfrak{B}} 

\newcommand{\fref}[1]{Figure~\ref{#1}} 
\newcommand{\ie}{\textit{i.e.}~} 
\newcommand{\ordinal}{\textsuperscript{th}} 
\definecolor{darkblue}{rgb}{0,0,0.7} 
\newcommand{\darkblue}{\color{darkblue}} 
\newcommand{\defn}[1]{\emph{\darkblue #1}} 

\begin{document}

\begin{abstract}
We show that the vertex barycenter of generalized associahedra and permutahedra coincide for any finite Coxeter system.
\end{abstract}

\maketitle

\vspace*{-20pt}

\tableofcontents

\vspace*{-30pt}

\section{Introduction}

Generalized associahedra were originally defined and studied by S.~Fomin and A.~Zelevinsky~\cite{FominZelevinsky-Ysystems} and by F.~Chapoton, S.~Fomin, and  A.~Zelevinsky~\cite{ChapotonFominZelevinsky}.
These polytopes realize finite type cluster complexes~\cite{FominZelevinsky-clusterAlgebrasII}.
More general polytopal realizations of these simplicial complexes were later constructed by C.~Hohl\-weg, C.~Lange, and H.~Thomas in~\cite{HohlwegLangeThomas} by removing boundary hyperplanes from Coxeter permutahedra.
This construction is based on Cambrian fans which were defined and studied by N.~Reading and D.~Speyer~\cite{Reading-CambrianLattices, Reading-sortableElements, ReadingSpeyer}.
Recently S.~Stella extended in~\cite{Stella} the approach of~\cite{ChapotonFominZelevinsky} and showed that the resulting realizations of generalized associahedra coincide with those of~\cite{HohlwegLangeThomas}.
In~\cite{PilaudStump-brickPolytope}, we provided a new approach to generalized associahedra using brick polytopes for spherical subword complexes.
We use this latter approach to prove that the vertex barycenters of all $c$-associa\-hedra coincide with the vertex barycenter of the underlying permutahedron.

\begin{theorem}
Let~$(W,S)$ be a finite Coxeter system, let~$c \in W$ be a Coxeter element, and let $\Asso[\basepoint]$ be the $c$-associahedron obtained from a fairly balanced $W$-permutahedron $\Perm[\basepoint]$ by removing all boundary hyperplanes not containing a $c$-singleton.
Then the vertex barycenters of $\Asso[\basepoint]$ and $\Perm[\basepoint]$ coincide.
\label{thm:main}
\end{theorem}

This property was observed by F.~Chapoton for Loday's realization of the classical associahedron~\cite{Loday}.
The balanced case in types $A$ and $B$ was conjectured by C.~Hohlweg and C.~Lange in~\cite{HohlwegLange} and proven by C.~Hohlweg, J.~Lortie, and A.~Raymond in~\cite{HohlwegLortieRaymond} using an orbit refinement of Theorem~\ref{thm:main} in these types.
The general balanced case was conjectured by C.~Hohlweg, C.~Lange, and H.~Thomas in~\cite{HohlwegLangeThomas}. This conjecture and the mentioned orbit refinement were again discussed in~\cite{Hohlweg}. We will as well settle this orbit refinement in the final Section~\ref{subsec:orbits}.

Finally, we want to remark that there is a naive approach to Theorem~\ref{thm:main}. Namely, one might hope that one can partition the vertices of the permutahedron $\Perm[\basepoint]$ in such a way that the sum of the vertices in each part corresponds to a vertex of the $c$-associahedron $\Asso[\basepoint]$.
There would even be a natural candidate for this partition coming from the theory of Cambrian lattices~\cite{Reading-sortableElements}.
This approach trivially works for dihedral types, but it already turns out to fail in type~$A_3$.
\vspace*{5pt}

\section{Generalized associahedra}
\label{sec:associahedra}


\subsection{Finite Coxeter groups}

Consider a \defn{finite Coxeter system}~$(W,S)$, acting on a real euclidean vector space~$V$ of dimension~$|S| = n$.
The \defn{Coxeter arrangement} is the collection of all reflecting hyperplanes in~$V$.
It decomposes~$V$ into open polyhedral cones whose closures are called \defn{chambers}.
The \defn{Coxeter fan} is the polyhedral fan formed by these chambers and all their faces.
We denote by~$\fundamentalChamber$ the \defn{fundamental chamber}, whose boundary hyperplanes are the reflecting hyperplanes of the reflections in~$S$.

Let $\Delta \eqdef \set{\alpha_s}{s \in S}$ denote the \defn{simple roots} and~$\nabla \eqdef \set{\omega_s}{s \in S}$ denote the \defn{fundamental weights}, defined such that~$s(\omega_t) = \omega_t-\delta_{s=t}\alpha_s$ for all~$s,t \in S$.
In other words, $\Delta$ and~$\nabla$ are dual bases of the euclidean space~$V$ (up to renormalization).
Geometrically, simple roots and fundamental weights give respectively normal vectors and rays of the fundamental chamber~$\fundamentalChamber$.
Let $\Phi \eqdef \set{ w(\alpha)}{ w \in W, \alpha \in \Delta}$ be the \defn{root system} for~$(W,S)$, with \defn{positive roots}~$\Phi^+ \eqdef \Phi \, \cap \, \R_{\ge 0}\Delta$ and \defn{negative roots}~$\Phi^- \eqdef -\Phi^+$.

A \defn{reduced expression} for an element~$w \in W$ is a minimal length expression of~$w$ as a product of generators in~$S$.
Let~$w_\circ$ denote the \defn{longest element} of~$W$, which sends the fundamental chamber~$\fundamentalChamber$ to its negative~$-\fundamentalChamber$.
We let~$\conjugation : S \to S$ denote its conjugation action on the generators defined by~$\conjugation(s) \eqdef w_\circ s w_\circ$.
Observe that~$w_\circ(\alpha_s) = -\alpha_{\conjugation(s)}$ and~$w_\circ(\omega_s) = -\omega_{\conjugation(s)}$.

We refer to~\cite{Humphreys, Humphreys1978} for further details on Coxeter groups.

\begin{example}[Symmetric groups]
\label{exm:symmetryGroups}
The symmetric group~$\fS_{n+1}$, acting on the hyperplane $H \eqdef \set{x \in \R^{n+1}}{\sum x_i = 0}$ by permutation of the coordinates, and generated by the adjacent transpositions~${\tau_p \eqdef (p \;\; p+1)}$, is the \defn{type~$A_n$ Coxeter system}.
Its simple roots are ${\Delta \eqdef \set{e_{p+1}-e_p}{p \in [n]}}$, its fundamental weights are $\nabla \eqdef \set{\sum_{q > p} e_q}{p \in [n]}$, and its root system is~$\Phi \eqdef \set{e_p-e_q}{p \neq q \in [n+1]}$.
Note that we have chosen the fundamental weights to match usual conventions, even if they do not live in the hyperplane~$H$.
The careful reader might prefer to project these weights down to~$H$ and adapt our discussions below accordingly.
Figures~\ref{fig:A2}\,(a) and~\ref{fig:A3}\,(a) show the type~$A_2$ and~$A_3$ Coxeter arrangements.

\begin{figure}
	\centerline{\includegraphics[width=\textwidth]{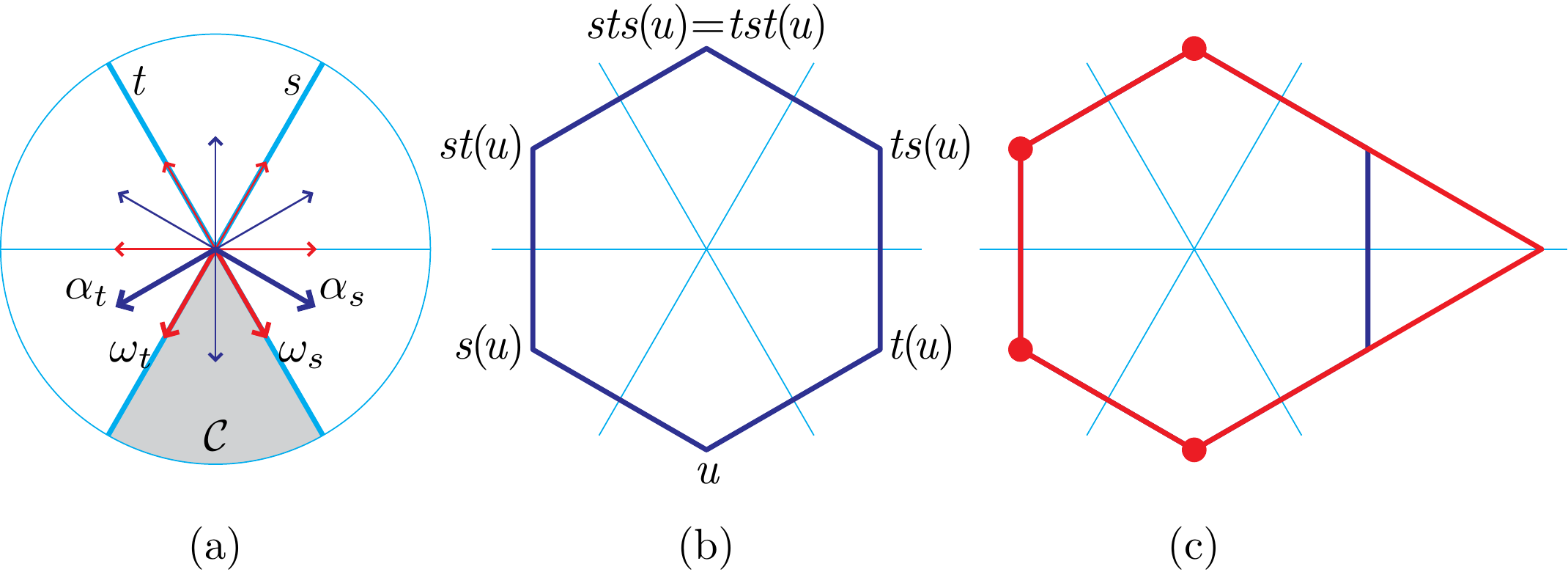}}
	\caption{The type~$A_2$ Coxeter arrangement (a) and the balanced $A_2$-permutahedron (b) and $(st)$-associahedron (c).}
	\label{fig:A2}
\end{figure}

\end{example}


\subsection{Permutahedra}

Let~$\basepoint$ be a point in the interior of the fundamental chamber.
We write~$\basepoint \eqdef \sum_{s \in S} \basepoint_s \omega_s$ with~$\basepoint_s \in \R_{> 0}$.
The \defn{$W$-permutahedron}~$\Perm[\basepoint]$ is the convex hull of the orbit of~$\basepoint$ under~$W$.
Its combinatorial properties are determined by those of the Coxeter group~$W$.
Let us observe in particular that:
\begin{enumerate}[(i)]
\item The normal fan of~$\Perm[\basepoint]$ is the Coxeter fan.
\item For any~$w \in W$ and~$s \in S$, the facet of~$\Perm[\basepoint]$ orthogonal to~$w(\omega_s)$ is defined by the inequality~$\dotprod{w(\alpha_s)}{x} \le \dotprod{\alpha_s}{\basepoint}$ and supported by the hyperplane~$w(\basepoint + \vect(\Delta \ssm \alpha_s))$.
\item Each face of~$\Perm[\basepoint]$ is a $W'$-permutahedron for a parabolic subgroup~$W'$ of~$W$, and the faces of~$\Perm[\basepoint]$ are naturally parametrized by cosets of parabolic subgroups of~$W$.
\end{enumerate}
We refer to the detailed survey on the $W$-permutahedra in~\cite{Hohlweg}.

If~$\basepoint = \one \eqdef \sum_{\omega \in \nabla} \omega$, we say that the $W$-permutahedron $\Perm[\basepoint]$ is \defn{balanced}, and we simply denote it by~$\Perm$ rather than~$\Perm[\one]$.
Note that the balanced $W$-permutahedron~$\Perm$ is (a translate of) the Minkowski sum of all positive roots (each considered as a one-dimensional polytope).
Finally, we say that the $W$-permutahedron~$\Perm$ is \defn{fairly balanced} if~$w_\circ(\basepoint) = -\basepoint$, \ie if $\basepoint_s = \basepoint_{\conjugation(s)}$ for all~$s \in S$.

\begin{example}[Classical permutahedron]
\label{exm:permutahedron}
The \defn{classical permutahedron} is the convex hull of all permutations of~$\{0,\dots,n\}$, regarded as vectors in~$\R^{n+1}$.
According to our choice of fundamental weights, we have~$\sum_{\omega \in \nabla} \omega = (0,\dots,n)$, so that the classical permutahedron coincides with the balanced $A_n$-permutahedron~$\Perm[][A_n]$.
Figures~\ref{fig:A2}\,(b) and~\ref{fig:A3}\,(b) show balanced $A_2$- and~$A_3$-permutahedra.
\end{example}


\subsection{Associahedra}
\label{subsec:generalizedassociahedra}

We now recall the construction of generalized associahedra given by C.~Hohlweg, C.~Lange, and H.~Thomas in~\cite{HohlwegLangeThomas}, based on the notions of sortable elements, Cambrian lattices and Cambrian fans defined and studied in~\cite{Reading-CambrianLattices, Reading-sortableElements, ReadingSpeyer}.

Fix a \defn{Coxeter element}~$c$ of~$W$, and a reduced expression~$\sq{c}$ of~$c$.
That is to say, $\sq{c}$ is a word on~$S$ where each simple reflection appears precisely once.
We say that~$s \in S$ is \defn{initial} in~$c$ if there is a reduced expression for~$c$ starting with~$s$.
For~$w \in W$, we denote by~$\sw{w}{c}$ the \defn{$\sq{c}$-sorting word} of~$w$, \ie the lexicographically first (as a sequence of positions) reduced subword of~$\sq{c}^\infty$ for $w$.
This word can be written as~$\sw{w}{c} = \sq{c}_{K_1}\sq{c}_{K_2}\cdots\sq{c}_{K_p}$, where~$\sq{c}_{K}$ denotes the subword of~$\sq{c}$ only taking the simple reflections in~$K \subset S$ into account.
The element~$w$ is then called \defn{$c$-sortable} if ${K_1\supseteq K_2\supseteq\cdots\supseteq K_p}$.
Observe that the property of being $c$-sortable does not depend on the particular reduced expression~$\sq{c}$ of the Coxeter element~$c$.
Finally, $w$ is called \defn{$c$-singleton} if $w$ is $c$-sortable and $w w_\circ$ is $(c^{-1})$-sortable.

The \defn{$c$-associahedron}~$\Asso[\basepoint]$ is obtained from $\Perm[\basepoint]$ by removing all boundary hyperplanes not containing any vector $w(\basepoint)$ for a $c$-singleton~$w$ of~$W$.
The boundary complex of its polar is the \defn{cluster complex} defined by S.~Fomin and A.~Zelevinsky in~\cite{FominZelevinsky-clusterAlgebrasII}, and its normal fan is the \defn{$c$-Cambrian fan} defined by N.~Reading and D.~Speyer~\cite{ReadingSpeyer}.
Note that the combinatorics of the $c$-associahedron (the cluster complex) does not depends on~$c$, while its geometry (in particular, the $c$-Cambrian fan) does.
We denote by~$\barycenter^\basepoint(c)$ the barycenter of the $c$-associahedron~$\Asso[\basepoint]$.
We say that $\Asso[\basepoint]$ is \defn{balanced} or \defn{fairly balanced} if $\Perm[\basepoint]$~is.
We simply denote by~$\Asso$ and~$\barycenter(c)$ the balanced $c$-associahedron and its barycenter.

\begin{example}[Loday's associahedron]
\label{exm:loday}
The cluster complex of type~$A_n$ is isomorphic to the simplicial complex of crossing-free sets of internal diagonals of a convex $(n+3)$-gon.
Its vertices are internal diagonals and its facets are triangulations of the $(n+3)$-gon.
In~\cite{Loday}, J.-L.~Loday provided an elegant realization of this simplicial complex, based on the following vertex description.
Label the vertices of the $(n+3)$-gon cyclically from~$0$ to~$n+2$.
Then, associate to each triangulation~$T$ of the $(n+3)$-gon its \defn{Loday vector}~$\loday(T)$ whose $j$\ordinal{} coordinate is given by:
$$\loday(T)_j \eqdef \big( j - \min\set{i \in [0,j-2]}{ij \in T} \big) \cdot \big( \max\set{k \in [j+2,n+2]}{jk \in T} - j \big).$$
\defn{Loday's associahedron} is the convex hull of the Loday vectors of all triangulations of the $(n+3)$-gon:
$$\loday_n \eqdef \conv \bigset{\loday(T)}{T \text{ triangulation of the } (n+3)\text{-gon}}.$$
In fact, Loday's associahedron~$\loday_n$ coincides with the balanced $(\tau_1 \cdots \tau_n)$-associa\-hedron $\Asso[][\tau_1 \cdots \tau_n][A_n]$.
This polytope is illustrated in Figures~\ref{fig:A2}\,(c) and~\ref{fig:A3}\,(c) for types~$A_2$ and~$A_3$.
Loday's vertex description of the associahedron was extended to a vertex description of all $c$-associahedra of type~$A$ and~$B$ in~\cite{HohlwegLange}.
\end{example}

\begin{figure}
	\centerline{\includegraphics[width=\textwidth]{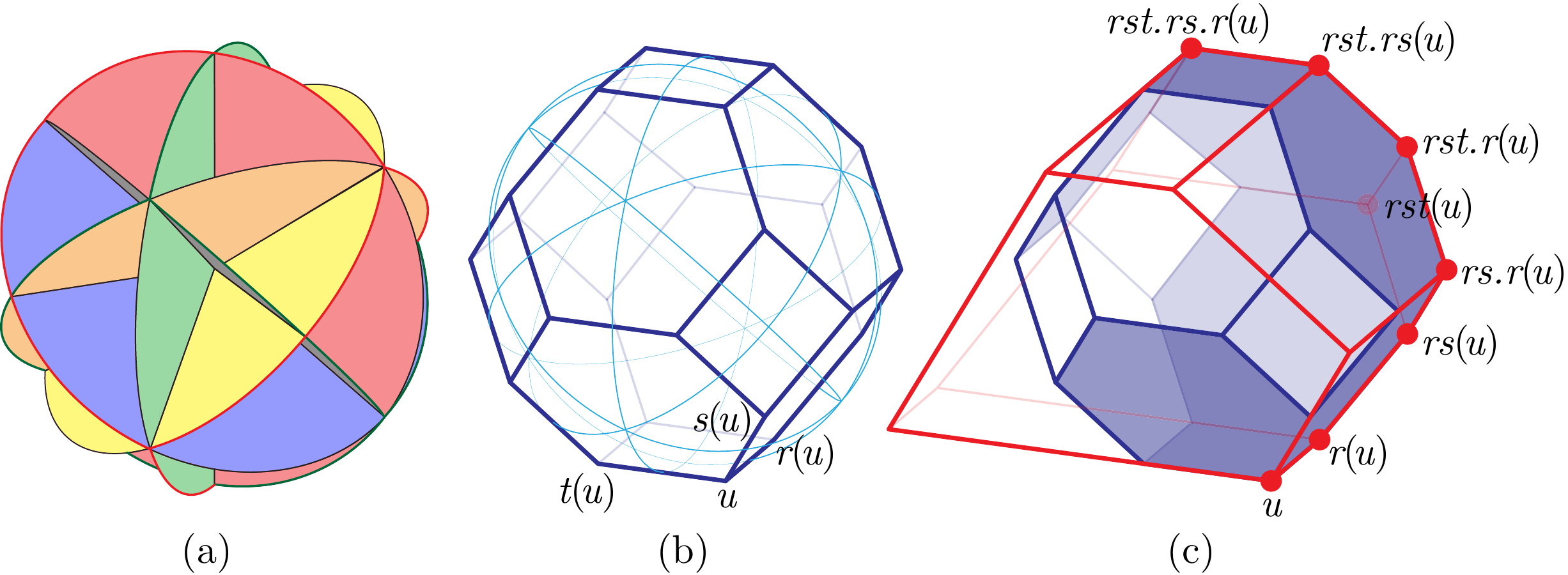}}
	\caption{The type~$A_3$ Coxeter arrangement (a) and the balanced $A_3$-permutahedron (b) and $(rst)$-associahedron (c).}
	\label{fig:A3}
\end{figure}


\subsection{Brick polytopes}

In the remainder of this section, we recall the viewpoint on $c$-associahedra coming from brick polytopes. We refer to~\cite{PilaudStump-brickPolytope} for the general treatment and to~\cite{PilaudSantos-brickPolytope} for a specific treatment of the type~$A$ situation.


\subsubsection{Subword complexes}

For a word $\Q \eqdef \q_1 \q_2 \cdots \q_m$ on~$S$, and an element~$\rho \in W$, A.~Knutson and E.~Miller define in~\cite{KnutsonMiller-subwordComplex} the \defn{subword complex}~$\subwordComplex$ to be the simplicial complex of those subwords of~$\Q$ whose complements contain a reduced expression for~$\rho$ as a subword.
A vertex of~$\subwordComplex$ is a position of a letter in~$\Q$.
We denote by~$[m] \eqdef \{1,2,\dots,m\}$ the set of positions in~$\Q$.
A facet of~$\subwordComplex$ is the complement of a set of positions which forms a reduced expression for~$\rho$ in~$\Q$.

In this paper, we only consider spherical subword complexes~$\subwordComplex$, for which we can assume that~$\rho = w_\circ$ and~$\Q$ contains a reduced expression for~$w_\circ$ (see~\cite[Theorem~3.7]{CeballosLabbeStump}).
We write~$\subwordComplex[\Q]$ instead of~$\subwordComplex[\Q,w_\circ]$ to shorten notations.

To avoid confusion between the words on~$S$ and the elements of~$W$, we use roman letters like~$\w \eqdef \w_1 \cdots \w_p$ for words and italic letters like~$w = w_1 \cdots w_p$ for group elements.
The distinction is in general clear from the context and can usually be ignored.

\begin{example}
\label{exm:clusterComplex}
Consider the Coxeter group~$\fS_3 = \gen{\tau_1,\tau_2} = \gen{s,t}$ and the word $\Q_2 \eqdef \sq{ststs}$.
The reduced expressions of~$w_\circ$ are~$sts$ and~$tst$.
Therefore, the facets of~$\subwordComplex[\Q_2]$ are~$\{1,2\}$, $\{2,3\}$, $\{3,4\}$, $\{4,5\}$, and~$\{1,5\}$, and the subword complex~$\subwordComplex[\Q_2]$ is a pentagon.
Similarly, for Coxeter group $\fS_4 = \gen{\tau_1,\tau_2,\tau_3} = \gen{r,s,t}$ and the word~$\Q_3 \eqdef \sq{rstrstrsr}$, the subword complex~$\subwordComplex[\Q_3]$ is isomorphic to the cluster complex of type~$A_3$.
\end{example}

\begin{example}[Primitive sorting networks]
\label{exm:networks}
In type~$A_n$, we can represent the word~$\Q = \q_1\q_2 \cdots \q_m$ by a \defn{sorting network}~$\cN_{\Q}$ as illustrated in \fref{fig:network}~(left).
The network~$\cN_\Q$ is formed by~$n+1$ horizontal lines (its \defn{levels}, labeled from bottom to top) together with~$m$ vertical segments (its \defn{commutators}, labeled from left to right) corresponding to the letters of~$\Q$.
If $\q_k = \tau_p$, the $k$\ordinal{} commutator of~$\cN_{\Q}$ lies between the $p$\ordinal{} and $(p+1)$\ordinal{} levels of~$\cN_{\Q}$.

\begin{figure}[h]
	\centerline{\includegraphics[width=\textwidth]{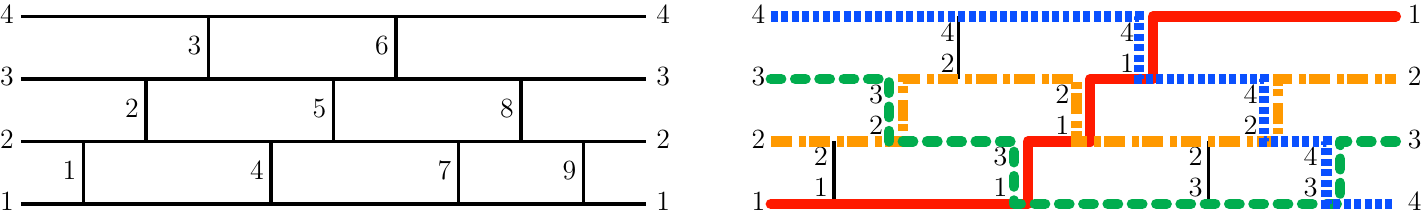}}
	\caption{The sorting network~$\cN_{\Q_3}$ and the pseudoline arrangement~$\Lambda_{\{1,3,7\}}$.}
	\label{fig:network}
\end{figure}

A \defn{pseudoline} supported by~$\cN_\Q$ is an abscissa monotone path on the network~$\cN_\Q$.
A commutator of~$\cN_\Q$ is a \defn{crossing} between two pseudolines if it is traversed by both pseudolines, and a \defn{contact} if its endpoints are contained one in each pseudoline.
A \defn{pseudoline arrangement}~$\Lambda$ (with contacts) is a set of~$n+1$ pseudolines on~$\cN_\Q$, any two of which have precisely one crossing, possibly some contacts, and no other intersection.
As a consequence of the definition, the pseudoline of~$\Lambda$ which starts at level~$p$ ends at level~$n-p+2$, and is called the $p$\ordinal{} pseudoline of~$\Lambda$.
As illustrated in \fref{fig:network} (right), a facet~$I$ of~$\subwordComplex[\Q]$ is represented by a pseudoline arrangement~$\Lambda_I$ supported by~$\cN_{\Q}$.
Its contacts (resp.~crossings) are the commutators of~$\cN_\Q$ corresponding to the letters of~$I$ (resp.~of the complement~of~$I$).

A \defn{brick} of~$\cN_{\Q}$ is a connected component of its complement, bounded on the right by a commutator of~$\cN_{\Q}$.
For $k \in [m]$, the $k$\ordinal{} brick is that immediately to the left of the $k$\ordinal{} commutator of~$\cN_{\Q}$.
\end{example}


\subsubsection{Brick polytopes}

In~\cite{PilaudStump-brickPolytope}, we constructed a polytope associated to a spherical subword complex $\subwordComplex[\Q]$ as follows.
To a facet~$I$ of~$\subwordComplex[\Q]$ and a position~$k$ in~$\Q$, we associate a root and a weight
$$\Root \eqdef \wordprod{\Q}{[k-1]\ssm I}(\alpha_{q_k}) \quad \text{and} \quad \Weight \eqdef \wordprod{\Q}{[k-1]\ssm I}(\omega_{q_k}),$$
where~$\wordprod{\Q}{X}$ denotes the product of the reflections~$q_x \in \Q$ for~$x \in X$.
The \defn{brick vector} of the facet~$I$ is the vector
$$\brickVector \eqdef \sum_{k \in [m]} \Weight,$$
and the \defn{brick polytope}~$\brickPolytope$ is the convex hull of all the brick vectors,
$$\brickPolytope \eqdef \conv \bigset{\brickVector}{I \text{ facet of } \subwordComplex[\Q]}.$$

\begin{example}[Counting bricks]
\label{exm:countingBricks}
In type~$A$, the above definitions for $\Root$, $\Weight$, and $\brickVector$ can be visually interpreted on the pseudoline arrangement~$\Lambda_I$ defined in Example~\ref{exm:networks}.
Namely, for any facet~$I$ of~$\subwordComplex[\Q]$, any position~$k \in [m]$, and any~$p \in [n+1]$,
\begin{enumerate}[(i)]
\item the root $\Root$ is the difference~$e_t-e_b$, where $t$~and~$b$ are such that the $t$\ordinal{} and $b$\ordinal{} pseudolines of~$\Lambda_I$ arrive respectively on top and bottom of the $k$\ordinal{} commutator of~$\cN_{\Q}$.
\item the weight~$\Weight$ is the characteristic vector of the pseudolines of~$\Lambda_I$ which pass above the $k$\ordinal{} brick of~$\cN_{\Q}$.
\item the $p$\ordinal{} coordinate of the brick vector~$\brickVector$ is the number of bricks of~$\cN_{\Q}$ below the $p$\ordinal{} pseudoline of~$\Lambda_I$.
\end{enumerate}
We refer to~\cite{PilaudSantos-brickPolytope} for further details on type~$A$ brick polytopes.
\end{example}

As observed in~\cite[Section~3.1]{CeballosLabbeStump}, the root function~${\Root[][\Q][\cdot][\cdot]}$ encodes the combinatorics of flips in the subword complex~$\subwordComplex[\Q]$.
We will mainly use here that~$\Root[][\Q][I][\cdot]$ is a bijection between the complement of~$I$ and~$\Phi^+$, and thus that
\begin{equation}
\label{eq:sumPositiveRoots}
\sum_{k \notin I} \Root = \sum_{\beta \in \Phi^+} \beta.
\end{equation}

In~\cite{PilaudStump-brickPolytope}, we proved that (the dual of) the brick polytope~$\brickPolytope$ realizes the subword complex if and only if the \defn{root configuration}~$\multiset{\Root}{i \in I}$ of any (or equivalently all) facet~$I$ of~$\subwordComplex[\Q]$ is linearly independent.
The following family is the main motivating example for the brick polytope construction.


\subsubsection{Generalized associahedra and brick polytopes}

Let~$c$ be a Coxeter element of~$W$, let~$\sq{c}$ be a reduced expression for~$c$, and let $\cwo{c} \eqdef \w_1 \cdots \w_N$ denote the $\sq{c}$-sorting word for the longest element~$w_\circ$.
According to~\cite{CeballosLabbeStump}, the subword complex $\clusterComplex \eqdef \subwordComplex[\cw{c}]$ is (isomorphic to) the cluster complex of type~$W$.
We furthermore proved in~\cite[Theorem~6.1]{PilaudStump-brickPolytope} that the brick polytope $\clusterPolytope \eqdef \brickPolytope[][\cw{c}]$ is indeed a polytopal realization of $\clusterComplex$.
Finally, we also proved in~\cite[Theorem~6.6]{PilaudStump-brickPolytope} that ---~up to affine translation by a vector~$\translation$~--- the brick polytope~$\clusterPolytope$ coincides with the balanced $c$-associahedron $\Asso$.
More explicitly, we have
\begin{equation}
\label{eq:associahedraBrickVersion}
\Asso = \conv \bigset{\brickVector[][\sq{c}][I] - \translation}{I \text{ facet of }\clusterComplex},
\end{equation}
where the affine translation $\translation$ is given by
$$\translation \eqdef \sum_{k \in [N]} w_1 \cdots w_{k-1}(\omega_{w_k}),$$
where $\w_1 \dots \w_N$ is the $\sq{c}$-sorting word for~$w_\circ$.
In Equality~(\ref{eq:associahedraBrickVersion}) and throughout the paper, we abuse notation and write~$\brickVector[][\sq{c}]$ rather than~$\brickVector[][\cw{c}]$, and similarly for~$\Root[][\sq{c}]$ and~$\Weight[][\sq{c}]$, as we already did for $\clusterComplex$ and for~$\clusterPolytope$.

\begin{example}
\label{exm:lodaycontinued}
The words~$\Q_2$ and~$\Q_3$ of Example~\ref{exm:clusterComplex} are precisely~$(\sq{st})\cwo{st}$ and~$(\sq{rst})\cwo{rst}$.
Therefore, the brick polytopes~$\brickPolytope[][\Q_2] = \clusterPolytope[][\sq{st}]$ and~$\brickPolytope[][\Q_3] = \clusterPolytope[][\sq{rst}]$ coincide, up to translation, with Loday's associahedra from Figures~\ref{fig:A2}\,(c) and~\ref{fig:A3}\,(c).

More generally, label the vertices of the $(n+3)$-gon cyclically from~$0$ to~$(n+2)$ and set~$\sq{c} \eqdef \tau_1 \tau_2 \cdots \tau_n$.
Consider the map sending the $i$\ordinal{} letter of~$\cw{c}$ to the $i$\ordinal{} internal diagonal of the $(n+3)$-gon in lexicographic order.
This map induces an isomorphism between the simplicial complex of crossing free sets of internal diagonals of the $(n+3)$-gon and the type~$A_n$ subword complex~$\clusterComplex$.
See~\cite{Woo, Stump, PilaudPocchiola} for details and extensions of this isomorphism.
The brick polytope~$\clusterPolytope$ then coincides with Loday's associahedron~$\loday_n$.
\end{example}


\subsubsection{Affine translation and greedy facets}

The \defn{positive greedy facet}~$\positiveFacet(\Q)$ (resp.~the \defn{negative greedy facet}~$\negativeFacet(\Q)$) is the lexicographically first (resp.~last) facet of~$\subwordComplex[\Q]$.
It turns out that $\positiveFacet(\Q)$ (resp. $\negativeFacet(\Q)$) is the unique facet whose root configuration $\multiset{\Root}{i \in I}$ contains only positive (resp. negative) roots.
These two particular facets were defined and studied in~\cite{PilaudStump-ELlabeling} to construct EL-labelings and canonical spanning trees for subword complexes.

We now focus on the situation of~$\clusterComplex$.
To simplify notation, we set here as well~$\positiveFacet_{\sq{c}} \eqdef \positiveFacet(\cw{c})$ and $\negativeFacet_{\sq{c}} \eqdef \negativeFacet(\cw{c})$.
Observe that $\positiveFacet_{\sq{c}}$ is the set of positions of the first appearance of the generators of~$S$ within~$\cw{c}$.
Similarly, $\negativeFacet_{\sq{c}}$ is the set of positions of the last appearance in~$\cw{c}$ of the generators of~$S$.
Up to transpositions of consecutive commuting letters, these are moreover the first and the last $n$ positions in~$\cw{c}$.

In Equality~(\ref{eq:associahedraBrickVersion}), the vertices of~$\clusterPolytope - \translation$ corresponding to the positive and negative greedy facets~$\positiveFacet_{\sq{c}}$ and~$\negativeFacet_{\sq{c}}$ coincide respectively with the vertices~$e(\one) = \one$ and~${w_\circ(\one) = -\one}$ of~$\Asso$. This implies that
\begin{equation}
\label{eq:greedy}
\translation = \brickVector[][\sq{c}][\positiveFacet_{\sq{c}}] - \one = \brickVector[][\sq{c}][\negativeFacet_{\sq{c}}] + \one.
\end{equation}


\subsubsection{The brick polytope for arbitrary basepoints}

We finally describe the situation of a general basepoint~$\basepoint \eqdef \sum_{s \in S} \basepoint_s \omega_s$ in the interior of the fundamental chamber.
As observed in~\cite[Remark~6.11]{PilaudStump-brickPolytope}, the brick polytope construction and its realization properties remain valid if we replace the root and weight functions by
$$\Root[\basepoint] \eqdef u_{q_k} \Root \quad \text{and} \quad \Weight[\basepoint] \eqdef u_{q_k} \Weight,$$
the brick vector by
$$\brickVector[\basepoint][\Q][I] \eqdef \sum_{k \in [m]} \Weight[\basepoint] = \sum_{k \in [m]} u_{q_k} \Weight,$$
and the brick polytope by
$$\brickPolytope[\basepoint] \eqdef \conv \bigset{\brickVector[\basepoint]}{I \text{ facet of } \subwordComplex[\Q]}.$$
The polytope~$\brickPolytope[\basepoint]$ is a deformation of~$\brickPolytope$.
Its combinatorics and its normal fan are controlled by the subword complex~$\subwordComplex[\Q]$, and therefore are independent of the basepoint~$u$, but its geometry (for example its edge lengths) is determined by~$u$.

For the subword complex~$\clusterComplex$, the polytope~$\clusterPolytope[\basepoint][c]$ is a translate of the $c$-associa\-hedron~$\Asso[\basepoint]$.
More precisely,
$$\Asso[\basepoint] = \conv \bigset{\brickVector[\basepoint][\sq{c}][I] - \translation[\basepoint][c]}{I \text{ facet of }\clusterComplex},$$
where the affine translation $\translation[\basepoint][c]$ is now given by
$$\translation[\basepoint][c] \eqdef \brickVector[\basepoint][\sq{c}][\positiveFacet_{\sq{c}}] - u = \sum_{k \in [N]} u_{q_k} w_1 \cdots w_{k-1}(\omega_{w_k}).$$


\section{The proof}

We first focus on balanced associahedra, and discuss the extension to fairly balanced associahedra in Section~\ref{subsec:fairlyBalanced}.
In the final Section~\ref{subsec:orbits}, we will discuss a further orbit refinement.
In view of \eqref{eq:associahedraBrickVersion}, the balanced version of Theorem~\ref{thm:main} reduces to the following~theorem.

\begin{theorem}
The vertex barycenter of the translated brick polytope~$\clusterPolytope - \translation$ coincides with that of the balanced $W$-permutahedron~$\Perm$. This is,
$$\sum \big( \brickVector[][\sq{c}][I] - \translation \big) = 0,$$
where the sum ranges over all facets $I$ of $\clusterComplex$.
\label{thm:revisedmain}
\end{theorem}

The proof of Theorem~\ref{thm:revisedmain} goes in two steps:
\begin{enumerate}[(i)]
\item We first prove that the barycenters of all balanced $c$-associahedra~$\Asso$ coincide, \ie that $\barycenter(c) = \barycenter(c')$ for any Coxeter elements~$c,c'$.
\item We then prove that the barycenter of the superposition of the vertex set of~$\Asso$ with the vertex set of~$\Asso[][c^{-1}]$ coincides with the origin, \ie that ${\barycenter(c) + \barycenter(c^{-1}) = 0}$.
\end{enumerate}
We will deduce these two statements from Lemmas~\ref{lem:commute}, \ref{lem:rotate}, \ref{lem:conjugate} and~\ref{lem:reverse}, which describe the impact on brick vectors of four natural operations on the word~$\Q$.


\subsection{Four operations}

In the next four lemmas, we study the behavior of the brick vectors of the facets of~$\subwordComplex[\Q]$ under four natural operations on the word~$\Q$, namely when we commute, rotate, conjugate, or reverse~$\Q$.
The first two operations were as well considered in~\cite[Propositions~3.8 and~3.9]{CeballosLabbeStump}.
Although we will only use them later for words of the form~$\cw{c}$, the statements below are valid for any word~$\Q$.

\begin{lemma}[Commute]
\label{lem:commute}
If~$\commute{\Q} \eqdef \q_{\pi(1)} \cdots \q_{\pi(m)}$ is obtained from~$\Q = \q_1 \cdots \q_m$ by a sequence of transpositions of consecutive commuting letters, then $\commute{}$ induces an isomorphism between the subword complexes $\subwordComplex[\Q]$ and $\subwordComplex[\commute{\Q}]$.
Moreover,
$$\brickVector[][\commute{\Q}][\commute{I}] = \brickVector$$
for any facet~$I$~of~$\subwordComplex[\Q]$.
\end{lemma}
\begin{proof}
The isomorphism is obtained directly from the definition of subword complexes, see~\cite[Proposition~3.8]{CeballosLabbeStump}.
Moreover, the definition of the weight function implies that
$$\Weight[][\commute{\Q}][\commute{I}][\commute{k}] = \Weight$$
for any facet~$I$ and position~$k$.
The result follows by summation.
\end{proof}

\begin{lemma}[Rotate]
\label{lem:rotate}
Let $\rotate{\Q} \eqdef \q_2 \cdots \q_m \conjugation(\q_1)$ be the \defn{rotation} of~$\Q = \q_1 \cdots \q_m$.
Then the cyclic rotation ${\rotation{} :i \mapsto (i-1)}$, where we identify~$0$ and~$m$, induces an isomorphism between the subword complexes $\subwordComplex[\Q]$ and $\subwordComplex[\rotate{\Q}]$.
Moreover,
$$\brickVector[][\rotate{\Q}][\rotation{I}] - \brickVector \; \in \; -2\omega_{q_1} + \R \alpha_{q_1}$$
for any facet~$I$ of~$\subwordComplex[\Q]$.
\end{lemma}

\begin{proof}
The isomorphism is again obtained directly from the definition of subword complexes, see~\cite[Proposition~3.9]{CeballosLabbeStump}.
From the definition of the weight function, we furthermore obtain that, for any facet~$I$ of~$\subwordComplex[\Q]$ and position~$k>1$,
$$\Weight[][\rotate{\Q}][\rotation{I}][\rotation{k}] = \begin{cases} \Weight & \text{if } 1 \in I, \\ q_1(\Weight) & \text{if } 1 \notin I, \end{cases}$$
and moreover,
\begin{equation}
\label{eq:firstfairlybalanced}
\Weight[][\rotate{\Q}][\rotation{I}][\rotation{1}] = \begin{cases} \ \,\phantom{\cdot \conjugation(q_1)}w_\circ(\omega_{\conjugation(q_1)}) = \Weight[][\Q][I][1] - 2 \omega_{q_1} & \text{if } 1 \in I, \\ w_\circ \cdot \conjugation(q_1) (\omega_{\conjugation(q_1)}) = \Weight[][\Q][I][1] - 2 \omega_{q_1} + \alpha_{q_1} & \text{if } 1 \notin I. \end{cases}
\end{equation}
The result follows by summation.
\end{proof}

\begin{lemma}[Conjugate]
\label{lem:conjugate}
Let $\conjugate{\Q} \eqdef \conjugation(q_1)\conjugation(q_2) \dots \conjugation(q_m)$ be the \defn{$w_\circ$-conjugate} of $\Q = \q_1 \dots \q_m$.
Then the subword complexes $\subwordComplex[\Q]$ and~$\subwordComplex[\conjugate{\Q}]$ coincide.~\mbox{Moreover,}
$$\brickVector[][\conjugate{\Q}][I] = -w_\circ(\brickVector)$$
for any facet~$I$ of~$\subwordComplex[\Q]$.
\end{lemma}

\begin{proof}
For any~$I \subset [m]$, we have $\wordprod{\conjugate{\Q}}{[m] \ssm I} = w_\circ \cdot \wordprod{\Q}{[m] \ssm I} \cdot w_\circ$, which ensures that $\subwordComplex[\Q] = \subwordComplex[\conjugate{\Q}]$.
Remembering that~$w_\circ(\omega_s) = -\omega_{\conjugation(s)}$, a direct calculation from the definition of the weight function gives
\begin{align}
\Weight[][\conjugate{\Q}]
& = w_\circ \cdot \wordprod{\Q}{[k-1] \ssm I} \cdot w_\circ (\omega_{\conjugation(q_k)}) \nonumber \\ 
& = -w_\circ \cdot \wordprod{\Q}{[k-1] \ssm I}(\omega_{q_k}) \label{eq:secondfairlybalanced}\\
& = -w_\circ (\Weight), \nonumber
\end{align}
for any facet~$I$ of~$\subwordComplex[\Q]$ and any position~$k$.
The result follows by summation.
\end{proof}

\begin{lemma}[Reverse]
\label{lem:reverse}
Let~$\reverse{\Q} \eqdef \q_m \cdots \q_1$ be the \defn{reverse} of~$\Q = \q_1 \cdots \q_m$.
Then the mirror ${\mirror{} : i \mapsto m-i+1}$ induces an isomorphism between the subword complexes $\subwordComplex[\Q]$ and $\subwordComplex[\reverse{\Q}]$.
Moreover,
$$\brickVector[][\reverse{\Q}][\mirror{I}] = w_\circ(\brickVector) + \sum_{\beta \in \Phi^+} \beta$$
for any facet~$I$ of~$\subwordComplex[\Q]$.
\end{lemma}

\begin{proof}
For any~$I \subset [m]$, we have $\wordprod{\reverse{\Q}}{[m] \ssm \mirror{I}} = (\wordprod{\Q}{[m] \ssm I})^{-1}$, which ensures that $\mirror{}$ is an isomorphism between the subword complexes $\subwordComplex[\Q]$ and $\subwordComplex[\reverse{\Q}]$.
Consider now a facet~$I$ of~$\subwordComplex[\Q]$. Since its complement in~$\Q$ forms a reduced expression for~$w_\circ$, we have
$$
w_\circ = \wordprod{\Q}{[m] \ssm I} =
\begin{cases}
	\wordprod{\Q}{[k-1] \ssm I} \cdot \wordprod{\Q}{[k+1, m] \ssm I} & \text{if } k \in I, \\
	\wordprod{\Q}{[k-1] \ssm I} \cdot q_k \cdot \wordprod{\Q}{[k+1, m] \ssm I} & \text{if } k \notin I.
\end{cases}
$$
Observe now that~$\wordprod{\Q}{[k+1, m] \ssm I} = \big(\wordprod{\reverse{\Q}}{[\mirror{k}-1] \ssm \mirror{I}} \big)^{-1}$.
This gives that for any position~$k \in I$, we have
\begin{align}
\Weight[][\reverse{\Q}][\mirror{I}][\mirror{k}]
& = \wordprod{\reverse{\Q}}{[\mirror{k}-1] \ssm \mirror{I}} (\omega_{q_k}) \nonumber \\
& = w_\circ \cdot \wordprod{\Q}{[k-1] \ssm I} (\omega_{q_k}) \label{eq:thirdfairlybalanced}\\
& = w_\circ(\Weight), \nonumber
\end{align}
and for any position~$k \in [m]\ssm I$, we have
\begin{align}
\Weight[][\reverse{\Q}][\mirror{I}][\mirror{k}]
& = \wordprod{\reverse{\Q}}{[\mirror{k}-1] \ssm \mirror{I}} (\omega_{q_k}) \nonumber \\
& = w_\circ \cdot \wordprod{\Q}{[k-1] \ssm I} \cdot q_k (\omega_{q_k}) \label{eq:fourthfairlybalanced}\\
& = w_\circ \cdot \wordprod{\Q}{[k-1] \ssm I}(\omega_{q_k}) - w_\circ \cdot \wordprod{\Q}{[k-1] \ssm I}(\alpha_{q_k}) \nonumber \\
& = w_\circ (\Weight) - w_\circ (\Root). \nonumber
\end{align}
Since $\sum_{k \notin I} \Root = \sum_{\beta \in \Phi^+} \beta$ (as we have seen in Equality~(\ref{eq:sumPositiveRoots})), and since $-w_\circ$ fixes $\sum_{\beta \in \Phi^+} \beta$, the result follows by summation.
\end{proof}

Combining Lemmas~\ref{lem:conjugate} and~\ref{lem:reverse}, we obtain the following corollary.
\begin{corollary}
\label{coro:conjugateReverse}
Denote by~$\conjugatereverse{\Q} \eqdef \conjugation(\q_m) \cdots \conjugation(\q_1)$ the reverse and $w_\circ$-conjugate of~$\Q = \q_1 \cdots \q_m$.
The mirror ${\mirror{} : i \mapsto m-i+1}$ defines an isomorphism between the subword complexes $\subwordComplex[\Q]$ and $\subwordComplex[\conjugatereverse{\Q}]$, and moreover
$$\brickVector + \brickVector[][\conjugatereverse{\Q}][\mirror{I}] = \sum_{\beta \in \Phi^+} \beta$$
for any facet~$I$ of~$\subwordComplex[\Q]$.
\end{corollary}

\begin{example}
The results of this section can be visually interpreted in type~$A$, using the sorting network interpretation of Example~\ref{exm:networks}.
Namely, fix a word~$\Q$ on the generators~$\{\tau_1, \dots, \tau_n\}$ of the type~$A_n$ Coxeter group, and a facet~$I$ of~$\subwordComplex[\Q]$.
Remember from Example~\ref{exm:countingBricks} that the $p$\ordinal{} coordinate of the brick vector~$\brickVector[][\Q][I]$ counts the number of bricks below the $p$\ordinal{} pseudoline of~$\Lambda_I$.
We have the following relations between the sorting networks~$\cN_\Q$, $\cN_{\rotate{\Q}}$, $\cN_{\conjugate{\Q}}$, and $\cN_{\reverse{\Q}}$.
\begin{enumerate}[(i)]
\item
The sorting network~$\cN_{\rotate{\Q}}$ is obtained from the sorting network~$\cN_\Q$ rotating its first commutator to the end.
The pseudolines of~$\Lambda_{\rotation{I}}$ coincide with that of~$\Lambda_I$, except the two pseudolines incident to the rotated commutator.
Therefore, the brick vector~$\brickVector[][\rotate{\Q}][I]$ also counts bricks of~$\cN_\Q$ below the pseudolines of~$\Lambda_I$.
The only difference concerns the contribution of the rotated brick.
This corresponds to Lemma~\ref{lem:rotate}.

\item
The sorting network~$\cN_{\conjugate{\Q}}$ is obtained~$\cN_\Q$ by a reflection with respect to the horizontal axis.
Therefore, the brick vector~$\brickVector[][\conjugate{\Q}][I]$ also counts bricks of~$\cN_\Q$ above the pseudolines of~$\Lambda_I$.
This corresponds to Lemma~\ref{lem:conjugate}.

\item
The sorting network~$\cN_{\reverse{\Q}}$ is obtained from~$\cN_\Q$ by a reflection with respect to the vertical axis.
Therefore, the brick vector~$\brickVector[][\reverse{\Q}][I]$ also counts bricks of~$\cN_\Q$ below the pseudolines of~$\Lambda_I$.
The only difference is that $\brickVector$ counts the leftmost unbounded bricks and not the rightmost unbounded bricks, while $\brickVector[][\reverse{\Q}][I]$ does the contrary.
This corresponds to Lemma~\ref{lem:reverse}.
\end{enumerate}
\end{example}


\subsection{Barycenter of balanced associahedra}

With the four preliminary lemmas of the previous section, we are now ready to prove Theorem~\ref{thm:revisedmain}.
As indicated earlier, we split the proof into two steps: we first prove that the barycenter~$\barycenter(c)$ of the $c$-associahedron~$\Asso$ is independent of the Coxeter element~$c$, and then that ${\barycenter(c) + \barycenter(c^{-1}) = 0}$.

\begin{proposition}
\label{prop:allEquals}
All balanced $c$-associahedra~$\Asso$ have the same vertex barycenter, \ie $\barycenter(c) = \barycenter(c')$ for any Coxeter elements~$c,c'$ of~$W$.
\end{proposition}

\begin{proof}
Fix a Coxeter element~$c$ and a reduced expression~$\sq{c} \eqdef \sq{c}_1 \cdots \sq{c}_n$ of~$c$.
Let~$c'$ denote the Coxeter element with reduced expression~$\sq{c'} \eqdef \sq{c}_2 \cdots \sq{c}_n \sq{c}_1$ obtained from~$\sq{c}$ rotating its first letter.
Up to commutations of consecutive commuting letters, the word~$\cw{c'}$ coincides with the word~$\rotate{\cw{c}}$, see~\cite[Proposition~4.3]{CeballosLabbeStump}.
Let~$\permutation{}$ denote the resulting isomorphism between $\clusterComplex$ and $\clusterComplex[c']$.
Lemmas~\ref{lem:commute} and~\ref{lem:rotate} ensure that
$$\brickVector[][\sq{c'}][\permutation{I}] - \brickVector[][\sq{c}][I] \; \in \; -2\omega_{c_1} + \R\alpha_{c_1},$$
for any facet~$I$ of~$\clusterComplex$.
Applying this to the positive greedy facet~$\positiveFacet_{\sq{c}}$ and using Equality~(\ref{eq:greedy}), we moreover obtain that
$$\translation[][c'] -\translation \; \in \; -2\omega_{c_1} + \R\alpha_{c_1}.$$
Therefore, ${(\brickVector[][\sq{c'}][\permutation{I}] - \translation[][c']) - (\brickVector[][\sq{c}][I] - \translation) \; \in \; \R\alpha_{c_1}}$ for any facet~$I$ of~$\clusterComplex$.
This translates on the barycenters to
\begin{equation}
\label{eq:rotation}
\barycenter(c') - \barycenter(c) \; \in \; \R\alpha_{c_1}.
\end{equation}
Consider now the sequence~$c^{(0)}, c^{(1)}, \dots, c^{(n)}$ of Coxeter elements obtained from~$c$ by repeatedly rotating the first letter.
That is to say,~$c^{(0)} \eqdef c = c_1 \cdots c_n$, then ${c^{(1)} \eqdef c' = c_2 \cdots c_n c_1}$, and in general~$c^{(k)} \eqdef c_{k+1} \cdots c_n c_1 \cdots c_k$. 
Since $c^{(0)} = c^{(n)}$, we have
$$\barycenter \big( c^{(1)} \big) - \barycenter \big( c^{(0)} \big) = \barycenter \big( c^{(1)} \big) - \barycenter \big( c^{(n)} \big) = \sum_{i \in [n-1]} \barycenter \big( c^{(i)} \big) - \barycenter \big( c^{(i+1)} \big) .$$
According to Equality~(\ref{eq:rotation}), the left-hand side belongs to the line~$\R\alpha_{c_1}$ while the right-hand side belongs to the vector space generated by~${\{\alpha_{c_2}, \dots, \alpha_{c_n}\} = \Delta \ssm \alpha_{c_1}}$.
Since~$\Delta$ is a linear basis of~$V$, this ensures that~$\barycenter(c') = \barycenter(c)$, \ie that  the barycenter is preserved by the rotation of the first letter.
Applying Lemma~\ref{lem:commute}, this remains true for the rotation of any initial letter in~$c$.
Since all Coxeter elements are related by repeated rotations of initial letters, the statement follows.
\end{proof}

\begin{proposition}
\label{prop:superposition}
The barycenter of the superposition of the vertices of the two associahedra $\Asso$ and $\Asso[][c^{-1}]$ is the origin, \ie $\barycenter(c) + \barycenter(c^{-1}) = 0$ for any Coxeter element~$c$.
\end{proposition}

\begin{proof}
Abusing notation, we write here $\sq{c}^{-1}$ for the reduced expression $\reverse{\sq{c}}$ of~$c^{-1}$.
As observed in~\cite[Remark~6.6]{CeballosLabbeStump}, the $(\sq{c}^{-1})$-sorting word of~$w_\circ$ is, up to transpositions of consecutive commuting letters, obtained by reversing and conjugating the $c$-sorting word of~$w_\circ$.
Denote by $\permutation{}$ the resulting isomorphism between $\clusterComplex$ and $\clusterComplex[c^{-1}]$.
From Lemma~\ref{lem:commute} and Corollary~\ref{coro:conjugateReverse}, we obtain that
$$\brickVector[][\sq{c}][I] + \brickVector[][\sq{c}^{-1}][\permutation{I}] = \sum_{\beta \in \Phi^+} \beta,$$
for any facet~$I$ of~$\clusterComplex$.
Observe now that $\permutation{}$ sends the positive greedy facet~$\positiveFacet_{\sq{c}}$ to the negative greedy facet~$\negativeFacet_{\sq{c}^{-1}}$.
Using Equality~(\ref{eq:greedy}), the previous equality for~$\positiveFacet_{\sq{c}}$ thus yields
$$\translation + \translation[][c^{-1}] = \sum_{\beta \in \Phi^+} \beta.$$
Thus, $(\brickVector[][\sq{c}][I] - \translation) + (\brickVector[][\sq{c}^{-1}][\permutation{I}] - \translation[][c^{-1}]) = 0$.
This translates on the barycenters~to
$$\barycenter(c) + \barycenter(c^{-1}) = 0,$$
which concludes the proof.
\end{proof}

\begin{proof}[Proof of Theorem~\ref{thm:revisedmain}]
For any Coxeter element~$c$ of~$W$, we obtain from Propositions~\ref{prop:allEquals} and~\ref{prop:superposition} that ${2 \barycenter(c) = \barycenter(c) + \barycenter(c^{-1}) = 0}$.
\end{proof}


\subsection{Barycenter of fairly balanced associahedra}
\label{subsec:fairlyBalanced}

By slight modifications of the arguments in the previous two sections, we are now ready to prove Theorem~\ref{thm:main} in full generality.

\begin{proof}[Proof of Theorem~\ref{thm:main}]
Let $\basepoint$ be a basepoint within the fundamental chamber $\fundamentalChamber$ for which $w_\circ(\basepoint) = -\basepoint$, \ie $\basepoint_s = \basepoint_{\conjugation(s)}$ for all~$s \in S$. By a careful analysis of the proofs of the four Lemmas~\ref{lem:commute}--\ref{lem:reverse} we can see that they all remain valid in slightly modified forms.
Lemma~\ref{lem:commute} stays valid as it is.
In Lemma~\ref{lem:rotate}, Equality~\eqref{eq:firstfairlybalanced} in its proof implies that the summand $-2\omega_{q_1}$ must be replaced by $-2\basepoint_{q_1}\omega_{q_1}$.
Lemma~\ref{lem:conjugate} stays valid, though we use in Equality~\eqref{eq:secondfairlybalanced} that $\basepoint_{q_k} = \basepoint_{\conjugation(q_k)}$.
Lemma~\ref{lem:reverse} stays as well valid, where $\basepoint_{q_k} = \basepoint_{\conjugation(q_k)}$ is used in Equalities~\eqref{eq:thirdfairlybalanced} and~\eqref{eq:fourthfairlybalanced}.
Using this, we finally obtain that Propositions~\ref{prop:allEquals} and~\ref{prop:superposition} hold as well in the case of fairly balanced associahedra. This completes the proof Theorem~\ref{thm:main} in full generality.
\end{proof}


\subsection{Orbit barycenter}
\label{subsec:orbits}

In Problem~3.4 in~\cite{Hohlweg} and the preceding discussion, C.~Hohlweg remarks that the vertex barycenter construction in~\cite{HohlwegLortieRaymond} depends on the fact that the rotation and reflection symmetries of the underlying convex polygon do not change the barycenter of the associahedron in type $A$.
As we have seen above, Proposition~\ref{prop:allEquals} and Proposition~\ref{prop:superposition} play analogous roles in the case of general finite Coxeter systems.
Denote by $\shift{c}$ the bijection of the facets of~$\clusterComplex$ induced by the rotation of all letters in $\sq{c}$.
In~\cite[Theorem~8.10]{CeballosLabbeStump} it is shown that in type $A_{n}$ with $\sq{c} = \tau_1 \cdots \tau_n$ being the long cycle, the operation $\shift{c}$ corresponds to the cyclic rotation of the underlying $(n+3)$-gon.
Denote moreover by $\permutation{}_1$ an isomorphism between $\clusterComplex$ and $\clusterComplex[c^{-1}]$ obtained by rotation as in the proof of Proposition~\ref{prop:allEquals}, and by $\permutation{}_2$ an isomorphism between $\clusterComplex$ and $\clusterComplex[c^{-1}]$ obtained by reversing and $w_\circ$-conjugation as in the proof of Proposition~\ref{prop:superposition}.
Observe that again in type $A_{n}$ with $\sq{c} = \tau_1 \cdots \tau_n$, the operation $\permutation{}_2^{-1} \circ \permutation{}_1$ corresponds to a reflection of the $(n+3)$-gon.
For all finite Coxeter systems, we obtain the following theorem.
\begin{theorem}
\label{thm:orbittheorem}
  Let $\orbit$ be the $\shift{c}$-orbit of a facet $I$ of $\clusterComplex$, and let $\orbit'$ be the $\shift{c}$-orbit of the facet $\big(\permutation{}_2^{-1} \circ \permutation{}_1\big)(I)$ of $\clusterComplex$.
  Then
  $$\sum \big( \brickVector[][\sq{c}][I] - \translation \big) = 0,$$
  where the sum ranges over the orbit $\orbit$ if the two orbits $\orbit$ and $\orbit'$ are equal, or over the disjoint union $\orbit \sqcup \orbit'$ if they are different.
\end{theorem}
\begin{proof}
  Observe that, as in type $A_{n-1}$, we have that $\shift{c^{-1}} \circ \permutation{}_1 = \permutation{}_1 \circ \shift{c}$, and as well $\shift{c^{-1}} \circ \permutation{}_2 = \permutation{}_2 \circ \shift{c}$. The statement then follows directly from Propositions~\ref{prop:allEquals} and~\ref{prop:superposition}.
\end{proof}

\begin{example}
  We consider again the subword complex $\clusterComplex[\sq{c}] = \subwordComplex[\sq{rstrstrsr}]$ as in Example~\ref{exm:clusterComplex}. We have seen in Example~\ref{exm:lodaycontinued} that the map sending the $i$\ordinal{} letter of~$\cw{c}$ to the $i$\ordinal{} internal diagonal of the $(n+3)$-gon in lexicographic order induces an isomorphism between the simplicial complex of crossing free sets of internal diagonals of the $(n+3)$-gon and the type $A_n$ subword complex. The $14$ triangulations of the $6$-gon are shown in Figure~\ref{fig:orbits}, and the facets of $\clusterComplex$ corresponding to the first triangulations in each of the rotation orbits A--D are given~by
  \begin{align}
    \{1,2,3\}, \quad \{6,7,8\}, \quad \{5,6,7\}, \quad \{1,3,7\}. \label{eq:orbitexamples}
  \end{align}
  Next, observe that $\shift{c}$ is the bijection
  $\left(\begin{smallmatrix}
    1 & 2 & 3 & 4 & 5 & 6 & 7 & 8 & 9 \\
    6 & 8 & 9 & 1 & 2 & 3 & 4 & 5 & 7
  \end{smallmatrix}\right)$
  on the position of letters in $\Q_3$.
  It thus sends those $4$ facets in~\eqref{eq:orbitexamples} to the facets
  $$\{6,8,9\}, \quad \{3,4,5\}, \quad \{2,3,4\}, \quad \{4,6,9\},$$
  which exactly correspond to the second triangulation in each orbit.
  Finally, observe that $\permutation{}_2^{-1} \circ \permutation{}_1$ is the bijection 
  $\left(\begin{smallmatrix}
    1 & 2 & 3 & 4 & 5 & 6 & 7 & 8 & 9 \\
    4 & 2 & 9 & 1 & 8 & 7 & 6 & 5 & 3
  \end{smallmatrix}\right)$
  and sends the $4$ facets in~\eqref{eq:orbitexamples} to the facets
  $$\{2,4,9\}, \quad \{5,6,7\}, \quad \{6,7,8\}, \quad \{4,6,9\}.$$
  This gives that orbit A is mapped onto itself, orbits B and C are interchanged, and orbit D is again mapped onto itself, as desired.
  Therefore, the barycenter of the vertices in orbit A is the origin, as is the barycenter of the vertices in orbit D. Moreover, the sum of the barycenter of vertices in orbit B and orbit C is as well the origin.
\begin{figure}
	\centerline{\includegraphics[width=\textwidth]{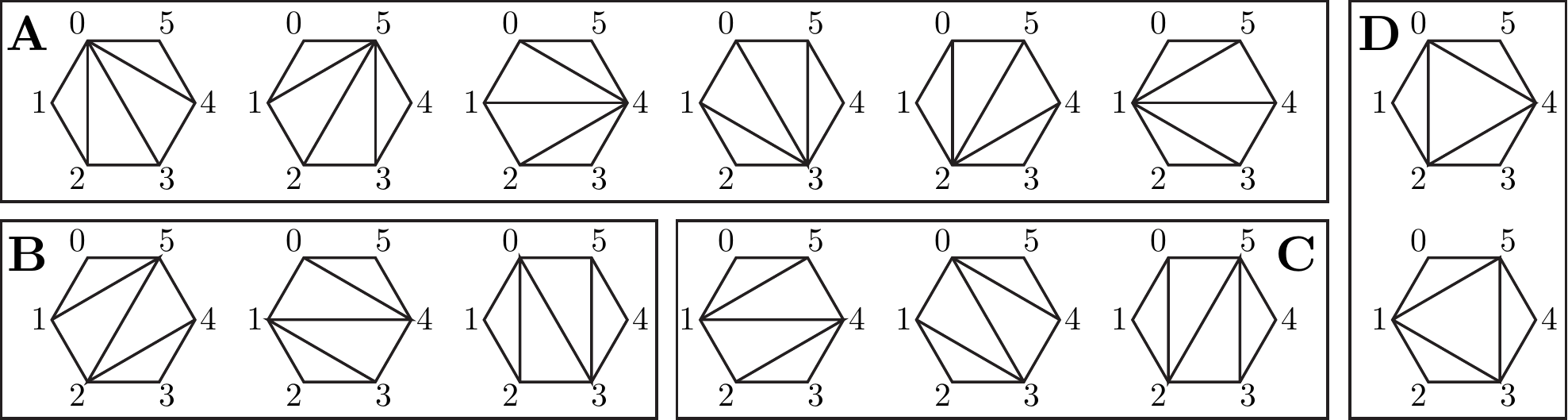}}
	\caption{The four orbits under rotation of the $14$~triangulations of the hexagon.}
	\label{fig:orbits}
\end{figure}

\end{example}


\section*{Acknowledgements}

  The authors would like to thank Christophe Hohlweg for bringing this open problem to their attention and for various discussions.

  We had the idea for the solution at the \emph{Formal Power Series and Algebraic Combinatorics} conference in Nagoya, Japan in August 2012. We are grateful to the organizers for this wonderful meeting and for the financial support.

\bibliographystyle{alpha}
\bibliography{PilaudStumpBarycenter}
\label{sec:biblio}

\end{document}